\documentclass[12pt,reqno]{amsart}
\usepackage{amsfonts}
\usepackage{amssymb,amsmath,amsthm, enumerate}
\usepackage{latexsym}
\usepackage{fullpage}
\usepackage{graphicx}
\usepackage{pkgfile}
\newtheorem{theorem}{Theorem}[section]

\newtheorem{definition}[theorem]{Definition}

\newtheorem{lemma}[theorem]{Lemma}
\newtheorem{coro}[theorem]{Corollary}

\theoremstyle{definition}
\newtheorem{remark}[theorem]{Remark}

\newcounter{tenumerate}

\def\P{\mathbb{P}}

\newcommand{\deq}{\stackrel{\scriptscriptstyle\triangle}{=}}

\renewcommand{\epsilon}{\varepsilon}

\DeclareMathOperator{\var}{Var} 
\newcommand{\remove}[1]{}
\renewcommand{\le}{\leqslant}
\renewcommand{\ge}{\geqslant}
\renewcommand{\leq}{\leqslant}
\renewcommand{\geq}{\geqslant}

\def\XXint#1#2#3{{\setbox0=\hbox{$#1{#2#3}{\int}$}
\vcenter{\hbox{$#2#3$}}\kern-.5\wd0}}

\begin{document}
\title{On level sets of Gaussian fields}
\author[S.\ Chatterjee]{Sourav Chatterjee$^*$}
\thanks{${}^*$Research partially supported by NSF grant DMS-1005312.}
\author[A.\ Dembo]{Amir Dembo$^\dagger$}\thanks{${}^\dagger$Research partially supported by NSF grant DMS-1106627.}
\author[J.\ Ding]{Jian Ding$^\ddagger$}\thanks{${}^\ddagger$Research partially supported by NSF grant DMS-1313596.} 

\address{$^*$ Stanford University, 
\newline\indent Stanford, California 94305}
\address{$^\dagger$ Stanford University, 
\newline\indent Stanford, California 94305}
\address{$^{\ddagger}$ University of Chicago,
\newline\indent Chicago, Illinois 60637}

\date{\today}

\subjclass[2010]{60G15, 60G70}

\keywords{Gaussian fields, extreme values, 
multiple valleys}

\maketitle

\maketitle

\begin{abstract}
In this short note, we present a theorem concerning certain ``additive structure'' for the level sets of non-degenerate Gaussian fields, which yields the multiple valley phenomenon for extremal
fields with exponentially many valleys.
\end{abstract}

\section{Introduction}
In this note we study the asymptotics of 
extreme values of \emph{mean zero} Gaussian fields.
To this end, a sequence of mean zero Gaussian fields $\{\eta_{n, v}: v\in V_n\}$, is called \emph{non-degenerate} if 
\beq
\lim_{n\to \infty} \frac{1}{\sigma_n} \E [\sup_{v\in V_n} \{\eta_{n, v}\}]  = \infty \mbox{ where } \sigma_n^2 \deq \sup_{v\in V_n}\{ \E [\eta_{n, v}^2]\}\,.
\eeq
A key role is played here by the functionals $g_n : 2^{V_n} \mapsto \mathbb{R}$ such that
\beq
g_n(S) = \E [\sup_{v\in S} \{\eta_{n, v}\}]  
\quad \mbox{ for all } \quad S \subseteq V_n\,.
\eeq
That is, $(g_n)$ are \emph{deterministic} functionals defined with respect to the laws of a specific sequence 
of Gaussian fields under consideration, with 
$g_n(S_n)=g_n(S_n(\omega))$ denoting the corresponding random variables in case of random sets $S_n(\omega)$.
Our main result is the following theorem revealing
the asymptotic additive structure of $g_n^2(\cdot)$ 
with respect to level sets of such fields.
\begin{theorem}\label{thm-1}
For any sequence of non-degenerate Gaussian fields $\{\eta_{n, v}: v\in V_n\}$ and any fixed number $0<\alpha<1$, define the level set
\beq\label{eq-level-size}
U_{n, \alpha} = \{v\in V_n: \eta_{n, v} \geq \alpha g_n(V_n)\}\,.
\eeq
Then, 
\beq\label{eq:sup-level}
\frac{g_n(U_{n, \alpha})}{g_n(V_n)} \to \sqrt{1 - \alpha^2}
\mbox{ in probability, when } n \to \infty.
\eeq 
\end{theorem}
\begin{remark}The connection between cover times and 
the extreme height squared 
of Gaussian free fields, hinging on Dynkin's isomorphism theorem \cite{Dynkin83, EKMRS00, MR92}) is  well-understood (see \cite{Ding11b, DLP10}). From this
perspective, Theorem~\ref{thm-1} is the analog of the following fact about cover times of random walks: 
On large finite graph with cover time $t_{\mathrm{cov}}$ (which is substantially larger than the maximal hitting time), run random walk $\alpha t_{\mathrm{cov}}$ steps, 
for some fixed $0<\alpha<1$. Then, by the additive structure of the random walk and concentration of the cover time around its mean (cf. \cite{Aldous91b}), it is not hard to show that 
the uncovered set by time $\alpha t_{\mathrm{cov}}$ has cover time close to $(1-\alpha) t_{\mathrm{cov}}$,
with high probability.  
\end{remark}

We provide next few applications of Theorem~\ref{thm-1}
in the context of Gaussian fields which upon proper normalization satisfy
\begin{equation}\label{eq-Gaussian-nice}
\var (\eta_{n, v}) \leq n \mbox{ for all } v\in V_n \mbox{ and } n^{-1} \log |V_n| \to  \log \lambda 
\mbox{ for fixed } \lambda > 1
\end{equation}
(indeed, our primary interest is in Gaussian fields 
having most of these variances close to $n$). Since 
\beq\label{eq:basic-bd}
g_n(S) \leq a_n + \sum_{v\in S} \int_0^\infty \P(\eta_{n, v} \geq a_n+x)dx \quad \mbox { for any } a_n \geq 0 
\quad \mbox {and all } \quad S \subseteq V_n \,,
\eeq
considering $a_n = \sqrt{2 n \log |S|}$ it follows from
our assumption \eqref{eq-Gaussian-nice} by elementary
Gaussian tail estimates, that
\begin{equation}\label{eq-union-bound}
g_n(S) \leq \sqrt{2 n \log |S|} + O(1)\,.
\end{equation}
Combining (\ref{eq-union-bound}) with Theorem~\ref{thm-1}, we get that
\begin{coro}\label{cor-multiple-vallyes}
For a sequence of Gaussian fields $\{\eta_{n, v}: v\in V_n\}$ satisfying \eqref{eq-Gaussian-nice} and 
any fixed $\beta>0$, there exists $c=c(\lambda, \beta)>0$ such that with probability tending to one as $n\to \infty$, we have
\beq\label{eq:large-level-set}
|\{v\in V_n: \eta_{n, v} \geq g_n(V_n) - \beta n\}| \geq \mathrm{e}^{cn}\,.
\eeq
\end{coro}
\begin{proof} Recall the Gaussian concentration inequality  
of Sudakov-Tsirelson \cite{ST74} and Borell \cite{Borell75}
\beq\label{eq:borel}
\P(|\sup_{v\in S_n} \{\eta_{n, v}\} - g_n(S_n)| \geq z)
\leq 2 \mathrm{e}^{-\frac{z^2}{2\sigma_n^2}}
\quad \mbox{ for all } z \geq 0
\eeq
(e.g., \cite[Thm. 7.1, Eq. (7.4)]{Ledoux89}), and
set $S_n \deq \{v\in V_n: \eta_{n, v} \geq g_n(V_n) - \beta n\}$ for Gaussian fields $\{\eta_{n,v}\}$ satisfying 
\eqref{eq-Gaussian-nice}. In case
$n^{-1} g_n(V_n) \leq \frac{\beta}{3}$, 
we get by the symmetry of the Gaussian law, upon 
considering \eqref{eq:borel} for $\{-\eta_{n,v}\}$
and $z=(\beta n - 2 g_n(V_n))_+$ that $S_n=V_n$
with probability tending to $1$ in $n$, so
(\ref{eq:large-level-set}) 
trivially holds. Assuming otherwise, without 
loss of generality we pass to a sub-sequence such 
that $n^{-1} g_n(V_n) \ge \frac{\beta}{3}$ for all $n$,
in which case the fields are non-degenerate and
with $S_n \supseteq U_{n,\alpha}$ for any
$\alpha \ge 1-\frac{\beta n}{3 g_n(V_n)}$ we have
by Theorem~\ref{thm-1} that 
$g_n(S_n) \geq (1-\alpha) g_n(V_n) \ge \frac{\beta n}{4}$ 
with probability tending to $1$ in $n$.  
Combining this with the upper bound of 
\eqref{eq-union-bound} completes the proof. 
\end{proof}

Recall the notion of multiple-valleys, a phenomenon of  interest in the study of spin glasses,
which is defined as follows (cf. \cite{Chatterjee08, Chatterjee09}).
\begin{definition}\label{dfn:mult-val}
A sequence of Gaussian field satisfying 
\eqref{eq-Gaussian-nice} exhibits multiple valleys if 
for any $\delta, \epsilon>0$, there exist $c=c(\delta, \epsilon, \lambda)>0$ and $W_n \subseteq V_n$ such that:
\begin{enumerate}[(a)]
\item $|W_n| \geq \mathrm{e}^{cn}$.
\item $\E (\eta_{n, v} \eta_{n, u}) \leq \epsilon \sqrt{\var( \eta_{n, v}) \var (\eta_{n, u})}$ for all $u, v\in W_n$. \label{item-orthogonal}
\item $\eta_{n, v} \geq g_n(V_n) - \delta n$ for all $v\in W_n$.
\end{enumerate}
\end{definition}

Corollary~\ref{cor-multiple-vallyes} tells us that any
sequence of Gaussian fields satisfying 
\eqref{eq-Gaussian-nice} shall exhibit sufficiently 
many high points to induce multiple valleys 
in the presence of the approximate orthogonality Condition~\eqref{item-orthogonal}. We note in passing
that \cite{Chatterjee09} shows that the SK spin-glass 
model exhibits a \emph{weak multiple valleys}  
phenomenon where the size of such $W_n$ grows 
as $(\log n)^{1/8}$ (as opposed to the exponential 
of $n$ size required in Definition \ref{dfn:mult-val}).

\begin{remark} 
Our assumption (\ref{eq-Gaussian-nice}) applies for 
the Gaussian fields
$\eta_{n, v} = \sum_{i=1}^n v_i Z_i$, where $\{Z_i\}$
are i.i.d.\ standard Gaussian variables and 
$v = (v_i)_i \in V_n \deq \{-1, 1\}^n$.
The maximum of such a Gaussian field is clearly  
$M_n \deq \sum_{i=1}^{n} |Z_i|$, 
achieved at $v^* = (\mathrm{sgn}(Z_i))_i$. Moreover, 
there exist fixed $\epsilon_0, \delta_0>0$ 
such that with high probability (as $n \to \infty$),
for any $v$ such that $\sum_i v_i v_i^* \le \epsilon_0 n$, 
we have $\eta_{n,v} \leq (1-\delta_0) M_n$.
\end{remark}

In view of the preceding example, 
additional structure is required for assuring
the orthogonality condition in the definition 
of multiple valleys. However, as we next show
this does apply for \emph{all extremal} 
Gaussian fields (as defined in \cite{Chatterjee08}).
That is, those 
(non-degenerate) Gaussian fields 
satisfying \eqref{eq-Gaussian-nice}, for which also
\beq\label{dfn:extremal}
\lim_{n\to \infty} \left\{\frac{g_n(V_n)}{n \sqrt{2 \log \lambda}}\right\} = 1\,.
\eeq
\begin{theorem}\label{thm-extremal}
Any sequence of extremal Gaussian fields 
$\{\eta_{n, v}: v\in V_n\}$ exhibits
multiple valleys as in Definition 
\ref{dfn:mult-val}, and moreover has level sets 
such that 
\beq\label{eq:level-set-size}
\lim_{n\to \infty} \frac{\log |U_{n, \alpha}|}{\log |V_n|}
 = 1-\alpha^2 \quad \mbox{ in probability, }
\quad \mbox{ for any fixed } \quad 0<\alpha<1.
\eeq
\end{theorem}

\begin{remark} Log-correlated Gaussian fields, 
including the two-dimensional Discrete 
Gaussian Free Field (\abbr{DGFF}, 
see \cite{BDG01}), are extremal Gaussian fields. 
For \abbr{DGFF} the asymptotics 
\eqref{eq:level-set-size} is derived
in \cite{Daviaud06}, with 
\cite{Chatterjee08} proving
the existence of weak multiple valleys 
having $|W_n|$ polynomial 
in $\log n$. 
\end{remark}

We conclude with three open problems of interest: 

\medskip
\noindent
$\bullet$ Find a necessary and sufficient condition 
for a sequence of Gaussian fields to be extremal. 

\medskip
\noindent
$\bullet$ Find an explicit condition on the 
covariance matrices which is equivalent to 
the corresponding Gaussian fields exhibiting 
multiple valleys.

\medskip
\noindent
$\bullet$ Extend our results to some non-Gaussian 
fields having tail and correlation structure similar
to those of \abbr{DGFF} or some other sequence
of extremal Gaussian fields.

\section{Proofs}

\begin{proof}[Proof of Theorem~\ref{thm-1}]
Let $\{\bar \eta_{n, v}: v\in V_n\}$ and $\{\tilde \eta_{n, v}: v\in V_n\}$ be two independent copies of the Gaussian field $\{\eta_{n, v}: v\in V_n\}$. Then, 
for any $0<\gamma < 1$,
\begin{equation}\label{eq-identity-in-law}
\{\eta_{n, v}: v\in V_n\}  \stackrel{law}{=} \{\gamma \bar \eta_{n, v} + \sqrt{1 -\gamma^2} \tilde \eta_{n, v} : v\in V_n\}\,.
\end{equation}
Let $\bar U_{n, t} = \{v\in V_n: \bar \eta_{n, v} \geq t g_n(V_n)\}$ for $0<t<1$ (so, clearly 
$\bar U_{n, t} \stackrel{law}{=} U_{n, t}$). Considering
\eqref{eq-identity-in-law} for $\gamma=t$ we get that
\beq
\sup_{v\in V_n} \{\eta_{n, v} \}\succeq t^2 g_n(V_n) {\bf 1}_{\{\bar U_{n, t} \neq \emptyset\}} + \sqrt{1 - t^2} \sup_{v\in \bar U_{n,t}} \{\tilde \eta_{n, v}\}\,,
\eeq
where $X \succeq Y$ means that $X$ stochastically 
dominates $Y$. Since the sequence of Gaussian fields is non-degenerate, by \eqref{eq:borel} we have that $g_n(V_n)^{-1} \sup_{v\in V_n} \{\eta_{n, v}\} \to 1$ in probability, hence
$\P(\bar U_{n, t} \neq \emptyset) \to 1$. Conditional on
$\bar U_{n,t}$, by the independence of
$\{\bar \eta_{n,v} \}$ and $\{ \tilde \eta_{n,v} \}$,
it further follows from \eqref{eq:borel} that 
\beq\label{eq:unif-tight}
\sigma_n^{-1} \Big(
\sup_{v\in \bar U_{n t}} \{\tilde \eta_{n, v} \} 
- g_n(\bar U_{n, t}) \Big)
\quad \mbox{ is a uniformly tight sequence. } 
\eeq
Therefore, for any fixed $\epsilon>0$ 
\begin{equation}\label{eq-upper-bound}
\limsup_{n\to \infty} \P\Big(g_n (\bar U_{n, t}) \geq (\sqrt{1 - t^2} + \epsilon)g_n(V_n)\Big) = 0\,,
\end{equation}
which upon taking $\epsilon \downarrow 0$ 
yields the upper bound in \eqref{eq:sup-level}. 
Next, fixing $\epsilon>0$ and considering 
\eqref{eq-identity-in-law} for $\gamma = \alpha$,
we get that
\begin{equation}\label{eq-domination}
\sup_{v\in V_n} \{\eta_{n, v} \} \preceq  \max_{t/\epsilon\in \{0, 1, \ldots, [1/\epsilon]-1\}} 
\Big\{ t \alpha g_n(V_n) 
+ \sqrt{1 - \alpha^2} \sup_{v\in \bar U_{n, t}} \{\tilde \eta_{n, v}\} + \epsilon g_n(V_n) \Big\}\,.\end{equation}
Further, per fixed $\{\bar \eta_{n,v}\}$ 
(and hence $\bar U_{n,t}$), we have by 
\eqref{eq:unif-tight} upon applying 
\eqref{eq:borel} for the Gaussian 
field on the LHS of \eqref{eq-domination}, 
that with probability tending to 1 as $n\to \infty$,
\beq\label{eq:t-alpha-bd}
g_n(V_n) \leq \max_{t/\epsilon \in \{0, 1, \ldots, [1/\epsilon]-1\}} \Big\{ t\alpha g_n(V_n) + \sqrt{1 - \alpha^2} g_n(\bar U_{n, t}) + 2 \epsilon g_n(V_n)\Big\}\,.
\eeq
Note that $h(t,\alpha) \deq \alpha t + \sqrt{1-\alpha^2}
\sqrt{1-t^2} \le 1$ for all $t,\alpha \in [0,1]$, with 
a strict inequality whenever $t \ne \alpha$. Hence, in
view of \eqref{eq:t-alpha-bd} and \eqref{eq-upper-bound}, 
there exist non-random $\delta_\epsilon \to 0$ as $\epsilon \to 0$, such that with 
probability tending to 1 as $n\to \infty$
$$
g_n(V_n) \leq \max_{t/\epsilon \in \{0, 1, \ldots, [1/\epsilon]-1\}, |t-\alpha| \leq \delta_\epsilon} \Big\{
t\alpha g_n(V_n) + \sqrt{1 - \alpha^2} g_n(\bar U_{n, t}) + 2 \epsilon g_n(V_n) \Big\}\,.
$$
That is to say, setting 
$\psi_\epsilon(\alpha) \deq (1-\alpha^2)^{-1/2} 
(1-\alpha^2 - 2 \epsilon - \alpha \delta_\epsilon)$, 
we have for any fixed $\epsilon >0$
\begin{equation}\label{eq-amir}
\lim_{n\to \infty}\P \Big(g_n(\bar U_{n, \alpha - \delta_\epsilon}) \leq \psi_\epsilon (\alpha) g_n(V_n) \Big) = 0\,.\end{equation}
Considering \eqref{eq-amir} with $\alpha$ replaced by $\alpha+\delta_\epsilon$ we get that and all  
$\epsilon < \epsilon_0(\alpha)$,
$$
\lim_{n\to \infty}\P\Big(g_n(\bar U_{n, \alpha}) \leq \psi_\epsilon(\alpha + \delta_\epsilon) g_n(V_n) \Big) = 0
\,,
$$
yielding the lower bound 
in \eqref{eq:sup-level} since 
$\delta_\epsilon \to 0$ and
$\psi_\epsilon(\alpha) \to \psi_0(\alpha)
=\sqrt{1-\alpha^2}$ when $\epsilon \to 0$.
\end{proof}
\begin{proof}[Proof of Theorem~\ref{thm-extremal}] By Markov's inequality, the upper bound in 
\eqref{eq:level-set-size} is 
a straightforward consequence of 
$$
\E \big[ |U_{n, \alpha}| \big] = \sum_{v \in V_n}  
\P(\eta_{n, v} \geq \alpha g_n(V_n)) 
\leq \lambda^{(1+\delta-\alpha^2)n},
$$
which by \eqref{dfn:extremal} and standard 
Gaussian tail bounds, holds for any $\delta>0$ 
and all $n$ large enough. For the corresponding 
lower bound we merely combine \eqref{eq:sup-level} of 
Theorem~\ref{thm-1}, with \eqref{eq-union-bound} 
in case of $S = U_{n,\alpha}$. This further implies that
with high probability, an independent copy of the 
Gaussian field $\{\eta_{n,v}\}$ restricted to the 
(random) subset $U_{n, \alpha}$ is also an 
extremal \emph{Gaussian} field. Applying Lemma
\ref{lem-mult-val} for the latter extremal fields
results with existence of multiple valleys for 
the original fields $\{\eta_{n,v}: v \in V_n\}$.
\end{proof}

\begin{lemma}\label{lem-mult-val}
For any sequence of extremal Gaussian fields $\{\eta_{n, v}: v\in V_n\}$ and $\epsilon>0$, there exist
$c = c(\lambda, \epsilon)$ and $W_n \subseteq V_n$ 
such that $|W_n| \geq \mathrm{e}^{cn}$ and 
$\E (\eta_{n, u} \eta_{n, v}) \leq \epsilon 
\sqrt{\var(\eta_{n, v}) \var (\eta_{n, u})}
$ for all $u, v\in W_n$.
 \end{lemma}
 \begin{proof}
For each $n$ and $\epsilon>0$ let $W_n$
be a maximal $\epsilon$-net of $V_n$ based on balls 
$B_n(v, \epsilon) \deq \{u\in V_n: \E( \eta_{n, v} \eta_{n, u}) \geq \epsilon \sqrt{\var(\eta_{n, v}) \var (\eta_{n, u})}\}$.
That is,  
$V_n = \bigcup_{u \in W_n} B_n(u, \epsilon)$ and
$u' \notin B_n(u,\epsilon)$ for all $u, u' \in W_n$.  
We claim that for any $\delta>0$ and all $n$ large enough,
\begin{equation}\label{eq-Amir-coffee}
\sup_{v \in V_n} \, g_n(B_n(v, \epsilon)) 
\leq \sqrt{1 - \epsilon^2} \sqrt{2\log \lambda}  
(1+ \delta) n\,.
\end{equation}
To this end, consider the Gaussian field 
$$
\tilde \eta_{n,u} \deq 
\eta_{n,u} - \rho_{n, u} \eta_{n,v} \,,
\qquad u\in B_n(v, \epsilon)\,,
$$
where $\rho_{n, u} \deq \frac{\E (\eta_{n, u} \eta_{n, v})}{\var (\eta_{n, v})} \geq \epsilon \frac{\sqrt{\var (\eta_{n, u})}}{\sqrt{\var (\eta_{n, v})}}$, and thereby 
$$
\var (\tilde \eta_{n, u}) \leq (1 - \epsilon^2) \var (\eta_{n, u}) \leq (1-\epsilon^2) n \,.
$$
Clearly, 
\beq\label{eq:jian1}
\E \big[ \sup_{u\in B_{n}(v,\epsilon)} \{ \eta_{n,u} \} 
\big] \leq \E [\sup_{u\in B_{n}(v, \epsilon)} 
\{ \rho_{n,u} \eta_{n, v} \}] + \E [\sup_{u\in B_{n}
(v,\epsilon)} 
\{\tilde \eta_{n, v}\}]\,.
\eeq
By Cauchy-Schwartz, 
$|\rho_{n, u}| \leq \frac{\sqrt{\var \eta_{n, u}}}{\sqrt{\var \eta_{n, v}}} \leq \frac{\sqrt{n}}{\sqrt{\var \eta_{n, v}}}$, so the first term on the RHS of 
\eqref{eq:jian1} is bounded by
$\frac{\sqrt{n}}{\sqrt{\var \eta_{n, v}}} \E |\eta_{n, v}| \leq 4 \sqrt{n}$, whereas applying 
\eqref{eq-union-bound} for the Gaussian field 
$\{(1-\epsilon^2)^{-1/2} \tilde \eta_{n,u} : u\in B_n(v, \epsilon)\}$ 
bounds the second term on the RHS of \eqref{eq:jian1}
by $\sqrt{1-\epsilon^2} \sqrt{2\log \lambda} (1+\delta/2) n$ for all $n$ large enough. Altogether, this establishes
\eqref{eq-Amir-coffee}. Let now 
$$
X_{n, v} \deq 
\sup_{u\in B_n(v, \epsilon)} \{\eta_{n, u}\} 
- g_n(B_n(v, \epsilon)) \,,
$$
noting that 
by \eqref{eq-Amir-coffee} and the definition of $W_n$,
$$
g_n(V_n) \leq \sqrt{1 - \epsilon^2} \sqrt{2\log \lambda}
(1+\delta) n + \E [ \sup_{v\in W_n} \{X_{n, v}\} ]\,.
$$
By \eqref{eq:borel}, the tail of each variable $X_{n,v}$ 
is dominated by that of Gaussian with variance $n$. Hence,
using a bound of the form \eqref{eq:basic-bd} 
for $a_n = \sqrt{2 n \log |W_n|}$, we obtain that 
$$
\E [ \sup_{v\in W_n} \{ X_{n, v} \} ] \leq 
a_n + O(1)\,,
$$
from which we deduce that 
$g_n(V_n) \leq \sqrt{1 - \epsilon^2} \sqrt{2\log \lambda} (1+\delta) n + \sqrt{2 n \log |W_n|} + O(1)$. 
Contrasting this upper bound with our assumption (\ref{dfn:extremal})
yields an exponential in $n$ lower bound on $|W_n|$.
 \end{proof}

\def\cprime{$'$}


\begin{thebibliography}{10}

\bibitem{Aldous91b}
D.~J. Aldous.
\newblock Threshold limits for cover times.
\newblock {\em J. Theoret. Probab.}, 4(1):197--211, 1991.

\bibitem{BDG01}
E.~Bolthausen, J.-D. Deuschel, and G.~Giacomin.
\newblock Entropic repulsion and the maximum of the two-dimensional harmonic
  crystal.
\newblock {\em Ann. Probab.}, 29(4):1670--1692, 2001.

\bibitem{Borell75}
C.~Borell.
\newblock The {B}runn-{M}inkowski inequality in {G}auss space.
\newblock {\em Invent. Math.}, 30(2):207--216, 1975.

\bibitem{Chatterjee08}
S.~Chatterjee.
\newblock Chaos, concentration, and multiple valleys.
\newblock Preprint, available at \verb|http://arxiv.org/abs/0810.4221|.

\bibitem{Chatterjee09}
S.~Chatterjee.
\newblock Disorder chaos and multiple valleys in spin glasses.
\newblock Preprint, available at \verb|http://arxiv.org/pdf/0907.3381v4.pdf|.

\bibitem{Daviaud06}
O.~Daviaud.
\newblock Extremes of the discrete two-dimensional {G}aussian free field.
\newblock {\em Ann. Probab.}, 34(3):962--986, 2006.

\bibitem{Ding11b}
J.~Ding.
\newblock Asymptotics of cover times via gaussian free fields: bounded-degree
  graphs and general trees.
\newblock Annals of Probability, accepted.

\bibitem{DLP10}
J.~Ding, J.~R. Lee, and Y.~Peres.
\newblock Cover times, blanket times, and majorizing measures.
\newblock {\em Ann. of Math. (2)}, 175(3):1409--1471, 2012.

\bibitem{Dynkin83}
E.~B. Dynkin.
\newblock Local times and quantum fields.
\newblock In {\em Seminar on stochastic processes, 1983 ({G}ainesville, {F}la.,
  1983)}, volume~7 of {\em Progr. Probab. Statist.}, pages 69--83. Birkh\"auser
  Boston, Boston, MA, 1984.

\bibitem{EKMRS00}
N.~Eisenbaum, H.~Kaspi, M.~B. Marcus, J.~Rosen, and Z.~Shi.
\newblock A {R}ay-{K}night theorem for symmetric {M}arkov processes.
\newblock {\em Ann. Probab.}, 28(4):1781--1796, 2000.

\bibitem{Ledoux89}
M.~Ledoux.
\newblock {\em The Concentration of Measure Phenomenon}, volume~89 of {\em
  Mathematical Surveys and Monographs}.
\newblock American Mathematical Society, Providence, RI, 2001.

\bibitem{MR92}
M.~B. Marcus and J.~Rosen.
\newblock Sample path properties of the local times of strongly symmetric
  {M}arkov processes via {G}aussian processes.
\newblock {\em Ann. Probab.}, 20(4):1603--1684, 1992.

\bibitem{ST74}
V.~N. Sudakov and B.~S. Tsirel{\cprime}son.
\newblock Extremal properties of half-spaces for spherically invariant
  measures.
\newblock {\em Zap. Nau\v cn. Sem. Leningrad. Otdel. Mat. Inst. Steklov.
  (LOMI)}, 41:14--24, 165, 1974.
\newblock Problems in the theory of probability distributions, II.

\end{thebibliography}
\end{document}